\documentclass[a4paper]{amsart}
\usepackage[english]{babel}
\usepackage[utf8]{inputenc}
\allowdisplaybreaks  

\usepackage{amssymb,xypic,amscd}

\usepackage{mathtools}
\usepackage{etoolbox}
\usepackage[colorlinks,linkcolor=Red3,citecolor=Red3,urlcolor=Cyan3]{hyperref}

\usepackage[svgnames,x11names,dvipsnames]{xcolor}
\usepackage[textsize=Small]{todonotes}  

\makeatletter
\providecommand\@dotsep{5}
\makeatother

\newtheorem{defn}[equation]{Definition}
\newtheorem{thm}[equation]{Theorem}

\newtheorem{cor}[equation]{Corollary}
\newtheorem{lem}[equation]{Lemma}
\newtheorem{prop}[equation]{Proposition}

\theoremstyle{remark}
\newtheorem{rem}[equation]{Remark}
\theoremstyle{remark}

\numberwithin{equation}{section} 

\renewcommand\hom{\operatorname{Hom}}
\newcommand\Hco{\operatorname{H}}

\newcommand\Pic{\operatorname{Pic}}
\newcommand\supp{\operatorname{supp}}

\newcommand\Div{\operatorname{Div}}

\newcommand\im{\operatorname{Im}}
\renewcommand\deg{\operatorname{deg}}
\renewcommand\emptyset{\varnothing}
\newcommand\perpdouble{\protect\mathpalette{\protect\perpDouble}{\perp}}
\def\perpDouble#1#2{\mathrel{\rlap{$#1#2$}\mkern3mu{#1#2}}}
\newcommand\pperp{\perpdouble}
\DeclarePairedDelimiterXPP\globalsymbol[1]{}{\langle}{\rangle}{}{\ifblank{#1}{\cdot,\cdot}{#1}}
\newcommand{\normSymbol}{N_{k(x)/k}}
\DeclarePairedDelimiterXPP\norm[1]{\normSymbol}{\lparen}{\rparen}{}{\ifblank{#1}{\cdot}{#1}}
\DeclarePairedDelimiterXPP\localsymbol[1]{}{\lparen}{\rparen}{}{\ifblank{#1}{\cdot,\cdot}{#1}}
\DeclarePairedDelimiterXPP\globalCommutator[1]{}{\lbrace}{\rbrace}{\raisebox{0pt}[1ex][1ex]{$^{\mathbb{A}_{X}}_{\mathbb{A}^{+}_{X}}$}}{\ifblank{#1}{\cdot,\cdot}{#1}}
\DeclarePairedDelimiterXPP\localCommutator[1]{}{\lbrace}{\rbrace}{^{K_{x}}_{\widehat{\O}_{X,x}}}{\ifblank{#1}{\cdot,\cdot}{#1}}

\renewcommand{\O}{{\mathcal{O}}}

\newcommand{\A}{{\mathbb{A}}}

\newcommand{\I}{{\mathbb{I}}}
\newcommand{\C}{{\mathbb{C}}}

\newcommand{\Z}{{\mathbb{Z}}}
\newcommand{\E}{\mathcal{E}}
\newcommand{\m}{\widehat{\mathfrak{m}}}

\renewcommand\tilde{\widetilde}

\title[An idelic quotient related to Weil reciprocity and the Picard group]{An idelic quotient related to Weil reciprocity and the Picard group}

\author[J.~M.~Mu\~noz Porras]{José María Mu\~noz Porras}
\email{jmp@usal.es}

\author[L.~M.~Navas]{Luis Manuel Navas Vicente}
\email{navas@usal.es}

\author[F.~Pablos~Romo]{Fernando Pablos Romo}
\email{fpablos@usal.es}

\author[F.~J.~PLaza~Mart\'in]{Francisco J. Plaza Mart\'in}
\email{fplaza@usal.es}

\thanks{Research of the first, third and fourth authors supported by grant MTM2015-66760-P of MINECO and SA030G18 of JCyL. Research of the second author supported by grant MTM2015-65888-C4-4-P (MINECO/FEDER)}

\address{Departamento de Matem\'aticas and IUFFYM, Universidad de
Salamanca,  Plaza de la Merced 1-4
        \\
        37008 Salamanca. Spain.
        \\
         Tel: +34 923294460. Fax: +34 923294583
}


\begin{document}

\maketitle

\begin{abstract}
This paper studies the function field of an algebraic curve over an arbitrary perfect field by using the Weil reciprocity law and topologies on the adele ring. A topological subgroup of the idele class group is introduced and it is shown how it encodes arithmetic properties of the base field and of the Picard group of the curve. These results are applied to study extensions of the function field.
\\
\center{\small{\textsc \it Dedicated to the memory of José María Muñoz Porras}
\\
}
\end{abstract}

\section{Introduction}
\label{sec:introduction}

The study of extensions of a given field is a classical problem in mathematics. Within algebraic number theory, class field theory is concerned with the classification of abelian extensions of local and global fields (\cite[Chap. XV, Tate's Thesis]{CaFr}, \cite{Weil-corps}).  

This paper studies function fields of algebraic curves over an arbitrary perfect field by mimicking the approach of class field theory. To be more precise, recall that, in the case of number fields, Tate proved the existence of a \textsl{duality} in the following sense: if $\Sigma$ is a number field and $\A_\Sigma$ is its adele ring, the character group of $\A_\Sigma/\Sigma$ is isomorphic to $\Sigma$. Here, we explore the multiplicative analog of this result in the case of function fields of an algebraic curve $X$. For this purpose, we consider the multiplicative group of the function field of the curve, $\Sigma^*_X$, as a subgroup of the idele group $\I^1_{X}$, and the pairing used to establish the duality is  the local multiplicative symbol (\cite{MP, Se}). In this way we are naturally led to consider the ``orthogonal,'' denoted by $(\Sigma_{X}^*)^{\pperp}$. Note that the inclusion $\Sigma_{X}^*\subseteq (\Sigma_{X}^*)^{\pperp}$ is equivalent to the Weil Reciprocity Law. 

We study the topology of the quotient group $(\Sigma_{X}^*)^{\pperp}/\Sigma_{X}^*$, show how it encodes arithmetical properties of the base field (see, for instance, Proposition~\ref{p:IX1/SigmaProfinite}), and describe it  in terms of the Picard group of $X$.  Among other results, the dependence of $(\Sigma_{X}^*)^{\pperp}/\Sigma_{X}^*$ with respect to the curve is best exhibited by our Theorem~\ref{thm:EXACTSEQSigma}, which shows the exactness of the sequence
\[
	0 \to \hom_{\text{gr}}\big(\Pic^0(X),k^{*}\big)  \to  
(\Sigma_{X}^{*})^{\pperp}/\Sigma_{X}^{*}  \to  
\Pic^0(X)
\]
for algebraically closed $k$ (and an analog for finite fields). In particular, it follows that $(\Sigma_{X}^*)^{\pperp}/\Sigma_{X}^*=0$ for $X={\mathbb{P}}_1$ over an algebraically closed field, although it may not vanish for other curves or other base fields. 

Finally, we apply our results to study field extensions $\Sigma_{X}\subset \Omega$ (see Theorems~\ref{T:equivalences} and~\ref{T:SigmaAlg}). In the near future, we aim at applying these techniques to characterize algebraic extensions and therefore to connect our approach with the geometric Langlands program. We also think that the study of algebraic extensions would benefit from the reinterpretation of our results in terms of $K$-theory, following the ideas of \cite{BryDeligne,TaSym} for the case of local and global fields.

\section{The adelic and idelic topologies}
\label{sec:topology}

This section aims to generalize some basic properties of adeles and ideles to the case of the function field of an algebraic curve over a perfect field, with special emphasis on algebraically closed base fields. Our approach follows Cassels' for global fields (\cite[Chp. II]{CaFr}) as well as other classical references (\cite{Se,WeilBas}). In the arithmetic framework of class field theory, the notion of global field comprises both finite extensions of ${\mathbb Q}$ as well as finite separable extensions of $\mathbb{F}(t)$, where $\mathbb{F}$ is a finite field and $t$ is trascendental over $\mathbb{F}$.

Somewhat surprisingly, there seems to be no comprehensive account of the adelic and idelic topologies in this more geometric context. Since the case of a global field over a finite base field is essentially disjoint from our main interests, we have felt the need to provide this initial discussion.

After recalling the basic definitions, we will establish the fundamental properties of the topologies on the ring of adeles and group of ideles associated to the field of rational functions of a smooth curve. The most well-known application of these adeles is probably Weil's proof of the Riemann-Roch Theorem, although Weil and many authors use \emph{repartitions}, sometimes also called \emph{pre-adeles}, which are defined using the local rings directly and not their completions.

\subsection{Definitions}
\label{sec:definitions}

Let $X$ be a smooth, complete and connected curve over a perfect field $k$. For simplicity, we will assume $k$ is algebraically closed in the function field of $X$ and also that $X$ has a $k$-rational point. We will use the following notation throughout:
\begin{itemize}

\item $\Sigma_{X}$ is the function field of $X$.

\item The notation $x \in X$ will always denote a \emph{closed} point of the curve $X$.

\item For $x \in X$, $\O_{X,x}$ is the valuation ring at $x$.

\item $\widehat{\O}_{X,x}$ is the completion of $\O_{X,x}$ with respect to its maximal ideal $\mathfrak{m}_{x}$.

\item $\m_{x}$ denotes the maximal ideal of the completion $\widehat{\O}_{X,x}$.

\item $K_{x}$ is the quotient field of $\widehat{\O}_{X,x}$.

\item For $x \in X$, let $k(x) = \widehat{\O}_{X,x}/\m_{x}$ be the residue field of $x$. Since $k$ is perfect, $k(x)/k$ is a finite separable extension. We let $\deg(x) = \text {dim}_{k} k(x)$.

\item For any ring $R$, its group of invertible elements will be denoted by $R^{*}$.

\end{itemize}

Recall that a perfect ground field $k$ implies that the closed points $x \in X$ are in one-to-one correspondence with the discrete valuations $v_{x}$ of the function field $\Sigma_{X}$, and $K_{x}$ is the completion of $\Sigma_{X}$ with respect to $v_{x}$. A choice of uniformizing parameter $t_{x}$ at $x$ determines an isomorphism of $\widehat{\O}_{X,x}$ and $K_{x}$ respectively with the power series ring and the field of Laurent series over the residue field, namely, $\widehat{\O}_{X,x} \simeq k(x)[[t_{x}]]$ and $K_{x} \simeq k(x)((t_{x}))$. In particular, $\widehat{\O}_{X,x}^{*} \simeq k(x)^{*} \times (1 + \m_{x})$. On occasion it will be convenient to use the same notation for an element of $k(x)^{*}$ and a lift to $\widehat{\O}_{X,x}^{*}$.

Recall that the ring of \emph{adeles} $\A_{X}$ of $\Sigma_{X}/k$ is the restricted direct product
\[
	   \A_{X} 
	:= \underset{x \text{ closed} } {\underset {x\in X} \prod'} \!\!\! K_{x}
	 = \left\{ (\alpha_{x})_{x\in X} \:\vert\: \alpha_{x}\in \widehat{\O}_{X,x}  \mbox{ for almost all } x\in X \right\}
\]
(where ``almost all'' means ``for all but finitely many''). For an adele $\alpha = (\alpha_{x})_{x\in X}$ and a (closed) point $x \in X$, $\alpha(x)$ will denote the image of its $x$-component $\alpha_{x}$ in the residue field $k(x)$.

We let $\A_{X}^+$ denote the subring of $\A_{X}$ given by the usual direct product
\[
	\A_{X}^+=\underset{x \text{ closed} } {\underset {x\in X} \prod}{\widehat
\O}_{X,x},
\]
i.e. where the word ``almost'' is dropped in the previous definition.
Finally, the \emph{idele group} $\I_{X}$ is the group $\A_{X}^{*}$ of invertible elements of $\A_{X}$. It is the restricted product of $K_{x}^{*}$ with respect to the unit groups $\widehat{\O}_{X,x}^{*}$.

\subsection{Topologies}

Given a divisor $D=\sum_{x\in X}n_{x}x$ on $X$, where $v_{x}$ is the valuation at $x \in X$ and $x \in X$ means $x$ runs over the closed points of $X$, the collection consisting of the $k$-vector subspaces
\[ 
	   U_{D}
	:= \big\{\alpha \in \A_{X}  \:\vert\:  v_{x}(\alpha_{x})\geq n_{x}  \ \forall x\in X \big\}
     = \!\!\!\! \prod_{x\in\supp  D} \!\!\! \m_{x}^{n_{x}}
	   \times
	   \!\!\!\! \prod_{x\notin\supp  D} \!\!\! \widehat{\O}_{X,x}
\] 
is a neighborhood base at the origin. We will refer to  the corresponding topology on $\A_{X}$ simply as \emph{the $\A_{X}$-topology}. It is easy to check that $\A_{X}$ becomes a topological ring when endowed with the $\A_{X}$-topology. The following properties are immediate:
\begin{enumerate}
	
\item $\A_{X}^{+}=U_{0}$.

\item $\A_{X} = \bigcup_{D} U_{D}$ and $(0) = \bigcap_{D} U_{D}$ where $D$ runs over $\Div(X)$.

\item If $D\geq D'$,  then $U_{D}\subseteq U_{D'}$ and $\dim_{k}(U_{D'}/U_{D})< \infty$.

\item Since $U_{D}$ is an open subgroup, it is also closed, thus clopen.

\end{enumerate}

\begin{prop}
\label{p:equivtops}
The  $\A_{X}$-topology coincides with the topology generated by the set of vector subspaces $C$ commensurable with $\A_{X}^+$, i.e., $\dim_{k}(C+\A_{X}^{+})/(C\cap \A_{X}^{+})<\infty$, as a neighborhood base of zero.
\end{prop}

\begin{proof}
Given $D\in \Div(X)$, it is easy to check that $U_{D}$ is commensurable with $\A_{X}^{+} = U_{0}$, i.e., $\dim_{k}(U_{D}+U_{0})/(U_{D}\cap U_{0})<\infty$, and hence $U_{D}$ is a neighborhood of $0$ in the $\A_{X}^{+}$-commensurable subspace topology.

Conversely, if $C$ is a subspace with $\dim_{k}(C+\A_{X}^{+})/(C\cap \A_{X}^{+})<\infty$ then, since $(0) = \bigcap_{D} U_{D}$, for some divisor $D\in \Div(X)$ we have $U_{D}\subset C$, thus $C$ is a neighborhood of $0$ in the $\A_{X}$-topology.
\end{proof}

For more details regarding commensurability, see for example~\cite{ACK}.

\begin{prop}
\label{P:discreteClosed}
Let $V\subseteq \A_{X}$ be a $k$-vector space. Then, $V$ has the discrete topology if and only if  $\dim_k (V\cap \A_{X}^+)<\infty$. If this is the case, then $V$ is closed. 
\end{prop}

\begin{proof}
$V$	is Hausdorff since $\{0\} = \bigcap_{D} U_{D}$. Hence, $V$ has the discrete topology if and only if $\{0\}$ is open, i.e. there exists $D>0$ such that $V\cap U_{D}=\{0\}$. Since $\dim_k \A_{X}^+/U_{D}<\infty$, the latter fact is equivalent to $\dim_k (V\cap \A_{X}^+)<\infty$. 

For the second part, note that the closure $\bar V$ is given by
\[
	\bar V = \bigcap_{D>0} (V+U_{D}).
\]
Assume that there is some element $w\in\bar V\setminus V$. Then for all $D$ there exists $v_{D}\in V$ with $w-v_{D}\in U_{D}$.  Discreteness implies that $V\cap U_{D_{0}}=\{0\}$ for some $D_{0}$, and thus $v_{D}$ is unique provided that $D\geq D_{0}$.  We claim that given $D>D_{0}$ there is some $D'>D$ such that $v_{D'}\neq 0$. Indeed, if this were not the case, then $w\in  \bigcap_{D'>D} U_{D'}=\{0\}$, which is a contradiction. Hence, we may fix $D_1>D_{0}$ such that $v_{D_1}\neq 0$. Since $\{0\} = \bigcap_{D_2>D_1} U_{D_2}$, we may choose $D_2>D_1$ such that $v_{D_1}\notin U_{D_2}$. The relations $w-v_{D_i} \in U_{D_i}$ for $i=1,2$, along with $U_{D_2}\subsetneq U_{D_1}$, yield  $v_{D_1}-v_{D_2}\in U_{D_1}$, and thus $v_{D_1}-v_{D_2}\in V\cap U_{D_1}=\{0\}$, which implies that $v_{D_1}\in U_{D_2}$, which is a contradiction.
\end{proof}


We now turn our attention to the ideles. The following discussion mirrors the ideas of \cite[Chp. II, \S 16]{CaFr} for global fields.

Since $\A_{X}$ is a topological ring, its group of invertible elements, the idele group $\I_{X}$, is endowed with the subset topology of the map
\begin{align*}
\I_{X} &\hookrightarrow \A_{X}\times \A_{X}\\
\alpha &\mapsto (\alpha ,\alpha^{-1})
\end{align*}
where $\A_{X}\times \A_{X}$ has the product topology. We shall refer to this topology as the \emph{$\I_{X}$-topology}. A basis of open neighborhoods of $1\in \I_{X}$ with respect to this topology consists of the sets
\begin{equation}
\label{E:basisIdelesVD}
	   V_{D}
	:= \!\!\! \prod_{x\in\supp D} \!\! (1+\m_{x}^{n_{x}})
	   \times
	   \!\!\!\!
	   \prod_{x\notin\supp D} \!\!\! \widehat{\O}^{*}_{X,x},	
\end{equation}
where $D=\sum_{x} n_{x} x$ is an \emph{effective} divisor on $X$. The $\I_{X}$-topology gives $\I_{X}$ the structure of a topological group. For this, merely restricting the $\A_{X}$-topology does not work. Analogously to the additive case, we will repeatedly make use of the following facts:
\begin{enumerate}
	

\item $\I_{X} \cap \A_{X}^{+} = V_{0} = \bigcup_{D} V_{D}$ and $(1) = \bigcap_{D} V_{D}$ where $D$ runs over effective divisors.

\item If $D \geq D'$, then $V_{D}\subseteq V_{D'}$.

\item Since $V_{D}$ is an open subgroup, it is also closed, thus clopen.

\end{enumerate}

On the other hand, a basic neighborhood of $1$ with respect to the $\A_{X}$-topology has the form $(1+U_{D})\cap\I_{X}$, where it can also be assumed that $D$ is effective. Now, given an effective divisor $D$, we have
\begin{equation}
\label{eq:VDinUD}
 	V_{D}\subseteq  (1+U_{D})\cap \I_{X},
\end{equation}
and thus the inclusion $\I_{X}\hookrightarrow \A_{X}$ is continuous with respect to the corresponding topologies.

Still following~\cite[Chp. II, \S 16]{CaFr}, where the \emph{content} of an idele is defined, we introduce a subgroup $\I_{X}^{1}\subset \I_{X}$ and study its topology.

There is a natural way to associate a divisor to an idele. Let $x\in X$ be a (closed) point and let $v_{x}$ be its associated valuation. Then the map
\begin{align}
\label{eq:IXaDivisor}
	\alpha & \mapsto D(\alpha):=\sum_{x\in X} v_{x}(\alpha_{x}) \, x :
	\quad \I_{X} \to \Div(X)
\end{align}
is a group homomorphism.

\begin{defn}
We define $\I_{X}^1$ as the ideles whose associated divisor via the map~\eqref{eq:IXaDivisor} has null degree:
\[
	   \I_{X}^1 
	:= \big\{\alpha\in \I_{X} \:\vert\: \operatorname{deg} D(\alpha):=\sum_{x\in X}\deg(x)v_{x}(\alpha_{x})=0 \big\}
\]
where $\deg(x)$ is the degree of the extension $k\hookrightarrow k(x)$.
\end{defn}

The following is the analog of~(\cite[Chp. II, p. 69, second lemma]{CaFr}).

\begin{lem}
\label{lem:IX1top}
On $\I_{X}^1$, the $\A_{X}$-topology and the $\I_{X}$-topology coincide.
\end{lem}

\begin{proof}
It is enough to compare the basic neighborhoods of  $\alpha=1\in \I_{X}^{1}$ with respect to both topologies. Recalling~\eqref{eq:VDinUD}, it suffices to show that
\[
 	(1+U_{D})\cap \I_{X}^{1} \subseteq V_{D}
\]
for an effective divisor $D=\sum_{x} n_{x} x$. By definition,
\begin{align*}
	U_{D} &= \left\{
	       \alpha \in \A_{X}
		   \bigg\vert
		   \begin{aligned}
		   v_{x}(\alpha_{x})\geq n_{x} &\text{ if $x\in \supp D$}
		                               \\
		   v_{x}(\alpha_{x})\geq 0     &\text{ if $x\notin \supp D$}
	\end{aligned} 
	\right\}
\intertext{and}
    V_{D} &= \left\{
	       \beta \in \I^1_{X}
		   \bigg\vert 
		   \begin{aligned}
		   v_{x}(\beta_{x}-1)\geq n_{x} &\text{ if $x\in \supp D$}
		                                \\
		   v_{x}(\beta_{x})= 0          &\text{ if $x\notin \supp D$}
           \end{aligned}
           \right\}.
\end{align*}
For $\alpha \in U_{D}$ such that $1+\alpha \in\I_{X}^{1}$ we have
\[
\begin{cases}
	      v_{x}((1+\alpha_{x}) -1)=v_{x}(\alpha_{x})\geq n_{x} 
	& \forall x\in \supp D,
	\\
	      v_{x}(1+\alpha_{x})\geq \operatorname{min} \{v_{x}(1),v_{x}(\alpha_{x})\}=0 
	& \forall x\notin \supp D.
\end{cases}
\]
Suppose that there is some $x_{0}\notin \supp D$ such that $v_{x_{0}}(1+\alpha)>0$. Since
\[
	\sum_{x\in X}v_{x}(1+\alpha_{x})=0,
\]
there must exist some $x_{1}\in \supp D$ such that $v_{x_{1}}(1+\alpha_{x})<0$. Since $v_{x_{1}}(1) = 0$, this would imply that $v_{x_{1}}(\alpha_{x})<0$, leading to a contradiction. Therefore $v_{x}(1+\alpha_{x})=0 $ for all $ x\notin \supp D$ and consequently $1+\alpha \in V_{D}$.
\end{proof}

For example, the kernel of the map~\eqref{eq:IXaDivisor} on the ideles is the open
neighborhood~\eqref{E:basisIdelesVD} associated to the zero divisor:
\[
	V_{0} = \{ \alpha \in \I_{X} : v_{x}(\alpha_{x}) = 0 \ \forall x\} 
	      = \prod_{x} \widehat{\O}^{*}_{X,x},
\]
and it is clearly a subset of $\I_{X}^{1}$, with $V_{0} = \I_{X}^{1} \cap \A_{X}^{+}$.
Note in particular that this implies $\I_{X}^{1}$ is an open subgroup of the idele group $\I_{X}$.

\begin{thm}
\label{thm:LocComIFFFinite}
\leavevmode
\begin{enumerate}

\item $\I_{X}^{1}$ is a Hausdorff topological group.

\item $\I_{X}^{1}$ is locally compact if and only if $k$ is finite. 

\item $\I_{X}^{1}$ is not locally connected.

\end{enumerate}
\end{thm}

\begin{proof}
That $\I_{X}^{1}$ is Hausdorff follows immediately from Lemma~\ref{lem:IX1top}. 

$\I_{X}^{1}$ is locally compact if and only if the open subset $V_{0}= \prod_{x\in X} \widehat{\O}_{X,x}^{*}$ is compact. If $k$ is finite, by Tychonoff's theorem, 
each $\widehat{\O}_{X,x}$, and hence also $V_{0}$, is compact. If $k$ is not finite, let $x$ be a closed point and consider the cover of $V_{0}$ given by the open subsets
\[
  (t+\m_{x})
  \times 
  \prod_{y\neq x}\widehat{\O}^{*}_{X,y},
\]
where $t$ runs over $k(x)^{*}$. Since $k^{*} \subseteq k(x)^{*}$ is infinite, this  cover admits no finite subcover.

For the third part, note that the above cover is disjoint, therefore $V_{0}$ is not connected precisely when $k(x)^{*}$ has more than one element for some $x$, which always clearly holds for $k$ infinite. For finite $k$ recall that, although the number of rational points of $X$ is finite, the number of closed points is infinite.
\end{proof}

\begin{defn}
We define
\[
	\bar{\I}_{X}^{1}:=  \I_{X}^{1} /\prod_{x\in X}(1+{\m}_{x})
\]
and endow it with the quotient topology via the projection ${\I}_{X}^{1} \overset{p}\to \bar{\I}_{X}^{1}$.
\end{defn}

Note that $\prod_{x} (1 + \m_{x}) = \bigcap_{x} V_{x}$, where $V_{x}$ is the neighborhood corresponding to $x$ regarded as a divisor, and thus $\bar{\I}_{X}^{1}$ is Hausdorff. However, $\prod_{x} (1 + \m_{x})$ may not be compact, and therefore the projection, although open, may not be a closed map.

\begin{thm}
\label{thm:barI-T2totdis}
A basis of neighborhoods of $1\in \bar{\I}_{X}^{1}$ is given by the subgroups
\[
    \bar{V}_{S} := \prod_{x\in S}\{1\} \times \prod_{x\notin S} k(x)^{*},
\]
where $S \subset X$ runs over the finite subsets. $\bar{\I}_{X}^{1}$ is Hausdorff and totally disconnected.
\end{thm}


\begin{proof}
Recall from  Lemma~\ref{lem:IX1top} that a basis of neighborhoods of $1\in \I_{X}^{1}$ is given by the collection of sets of the form
\[
	  V_{D} 
	= \prod_{x\in \supp D}(1+\m_{x}^{n_{x}})
	  \times
	  \prod_{x\notin \supp D} \widehat{\mathcal O}_{X,x}^{*},
\]
where $D=\sum_{x} n_{x}\cdot x $ is an effective divisor. Since $p$ is the quotient map, the images $p(V_{D})$ are an open basis of neighborhoods of $1$ in $\bar{\I}_{X}^{1}$. Thus, for the first part, it suffices to check that $p(V_{D}) = \bar{V}_{S}$, where $S = \supp(D)$. Since $V_{D}\subseteq V_{0}$ for $D \geq 0$,
\[
	p(V_{D}) = V_{D}/(V_{D}\cap \prod(1+{\m}_{x})) \hookrightarrow    
	p(V_{0}) = V_{0}/(V_{0}\cap \prod(1+{\m}_{x}))   = 
	\prod_{x\in X} k(x)^{*} \hookrightarrow
	\bar{\I}_{X}^{1}.
\]
Keeping in mind that
\[
	  V_{D} \cap \prod_{x}(1+{\m}_{x})
	= \prod_{x\in \supp D}(1+\m_{x}^{n_{x}})
	  \times\!\!\!\!
	  \prod_{x\notin \supp D} (1+\m_{x}),
\]
it follows that
\[
     p(V_{D})
   = V_{D}/(V_{D}\cap \prod(1+{\m}_{x}))
   = \prod_{x\in \supp D}\{ 1\}
     \times
     \prod_{x\notin \supp D} k(x)^{*}
   = \bar{V}_{S}.
\]

Since we already know $\bar{\I}_{X}^{1}$ is Hausdorff, it remains to check that $\bar{\I}_{X}^{1}$ is totally disconnected. First, the connected component of the identity, $C$, is contained in the intersection of all its clopen neighborhoods 
Thus
\[
	C
	\subseteq  
	\bigcap_{\substack{\bar{V} \subset \I_{X}^{1} \\ \text{clopen}}} \bar{V}
	\subseteq  
	\bigcap_{\substack{S\subset X \\ \text{finite} }} \bar{V}_{S} = \{1\}.
\]
Accordingly, since the quotient of a topological group modulo its identity component is always totally disconnected, this implies $\bar{\I}_{X}^{1}$ is totally disconnected.
\end{proof}

\subsection{The function field embedded in the adeles.}

We now consider how the function field $\Sigma_{X}$ and its unit group $\Sigma_{X}^{*}$, sit inside the adeles and ideles, respectively, via the diagonal embedding. The main result here is the Strong Approximation Theorem~(Theorem~\ref{T:StrongApprox}). The statement parallels~\cite{CaFr}, but note that $\Sigma_{X}$ is not a global field for an infinite base field $k$, and thus our Theorem~(Theorem~\ref{T:StrongApprox}) is not a consequence of the version in~\cite{CaFr}. Our proof uses more geometric language.

\begin{prop}
\label{prop:Sigmadiscrete}
The function field $\Sigma_{X}$ is closed and discrete in $\A_{X}$. 
\end{prop}

\begin{proof}
By Proposition~\ref{P:discreteClosed}, it is enough to show discreteness, which follows from the finite-dimensionality of the $k$-vector space $\A_{X}^+\cap \Sigma_{X}=\Hco^0(X,\O_{X})$.
\end{proof}

\begin{thm}[Strong Approximation Theorem]
\label{T:StrongApprox}
The diagonal embedding
\[
\Sigma_{X} \to {\underset {x\neq x_{0}}\prod'}K_{x}
\]
(the restricted product of $K_{x}$ with respect to the subrings $\widehat{\O}_{X,x}$) is injective with dense image in the topology induced by $\A_{X}$.
\end{thm}

\begin{proof}
We shall follow the proof of \cite[Chp. II, p. 67, Strong Approximation Theorem]{CaFr} using geometric language. The injectivity follows straightforwardly from the fact that $\Sigma_{X}\hookrightarrow K_{x}$ for all $x$. 

To see that the image is dense, we have to prove that every neighborhood of an arbitrary element $\bar\lambda \in \prod'_{x\neq x_{0}} K_{x} $ has nonempty intersection with $\Sigma_{X}$. A basis of open neighborhoods of $0$ in the induced topology consists of those subsets of the form $ \pi_{x_{0}}(U_{D})$ for an effective divisor $D$, where $\pi_{x_{0}}:\A_{X} \to \prod'_{x\neq x_{0}} K_{x}$ removes the $x_{0}$ component. 
Thus, it suffices to prove that $\Sigma_{X} \cap \big( \bar\lambda+  \pi_{x_{0}}(U_{D}) \big)\neq\emptyset$, which is tantamount to
\[
	\Sigma_{X} \cap   (\lambda  + U_{D} + K_{x_{0}})  \neq  \emptyset,
\]
where $\lambda $ is a preimage of $\bar\lambda$. Since $D$ can be made arbitrarily large, we can assume that $D(\lambda)\leq  D$. Moreover,
\[
	  U_{D} + K_{x_{0}} 
	= \hspace{-1em}
	  \prod_{x \in  \supp  D\setminus\{x_{0}\}}
	  \hspace{-1em}
	  {\m}_{x}^{n_{x}} \times K_{x_{0}}
	  \times
	  \hspace{-1.5em}
	  \prod_{x\notin \supp  D\cup\{x_{0}\}}
	  \hspace{-1em}
	  \widehat{\O}_{X,x},
\]
and thus we can also assume that $x_{0}\notin \supp  D$. The statement we want to prove is equivalent to the existence of a function $f\in\Sigma_{X}$ verifying:
\begin{equation}
\label{eq:fApprox}
\begin{gathered}
	\text{$v_{x}(f-\lambda) \geq n_{x}$ for all $x\in \supp  D $,}
	\\ 
	\text{$v_{x}(f)\geq 0$ for all $x\notin \supp   D\cup\{x_{0}\}$.}
\end{gathered}
\end{equation}
Let us consider the sequence
\[
	0\to \O_{X}(-D) \to \O_{X}(-D(\lambda)) \to \O_{X}(-D(\lambda))/\O_{X}(-D) \to 0.
\]
If we tensor this with $\O_{X}(n x_{0})$ for $n\gg 0$, we obtain the following surjection via the long exact cohomology sequence:
\[
	                             H^0(X, \O_{X}(nx_{0} -D(\lambda)) )
	\overset{p}{\longrightarrow} \bigoplus_{x\in \supp  D} \m_{x}^{v_{x}(\lambda)} / \m_{x}^{n_{x}} \to 0.
\]
Now, let $\hat \lambda$ be the class of $\lambda$ in $\bigoplus_{x\in X} \m_{x}^{v_{x}(\lambda)}  /  \m_{x}^{n_{x}} $ and let  $f\in p^{-1}(\hat\lambda) \subset\Sigma_{X}$ be a preimage by $p$ of $\hat\lambda$. This is the function we need.
First, note that $f-\lambda\in \A_{X}$ and the identity $p(f)-\hat\lambda=0$ yields
\[
	f-\lambda\in  \m_{x}^{n_{x}} \quad \forall x\in \supp  D.
\]
Finally, since $D(f)+nx_{0} -D(\lambda)\geq 0$, we have $v_{x}(f)\geq 0$ for $x\notin \supp   D\cup\{x_{0}\} $. 
\end{proof}

Compactness of the quotient $\I_{X}^{1}/\Sigma_{X}^{*}$ in the case of global fields is a basic result of class field theory, where $\I_{X}^{1}/\Sigma_{X}^{*}$ corresponds to the idele class group. Here we give an analogous general result as well as a converse.

\begin{thm}
\label{thm:PropertiesSigma*}
$\Sigma_{X}^{*}$ is discrete and closed as a subgroup of $\I_{X}^{1}$. The quotient $\I_{X}^{1}/\Sigma_{X}^{*}$ is compact if and only if $k$ is finite.
\end{thm}

\begin{proof}
Recall that the $\A_{X}$-topology and the $\I_{X}$-topology coincide on $\I_{X}^{1}$ (Lemma~\ref{lem:IX1top}). That $\Sigma_{X}^{*}$ is closed and discrete then follows immediately from this and Proposition~\ref{prop:Sigmadiscrete} by observing that $\Sigma_{X}^{*}= \Sigma_{X}\cap \I_{X}^{1}$.

For the second part, first assume that $k$ is finite. Fix an effective divisor $D$ with $\deg(D) \geq g$, the genus of $X$. Let us check that there is a continuous surjection
\[ 
	U_{-D}\cap \I_{X}^{1} \to  \I_{X}^{1}/\Sigma_{X}^{*}.
\]
Given $\lambda\in \I_{X}^{1}$, we have $\deg(D+D(\lambda))>g$ and accordingly $H^0(X,\O_{X}((D+D(\lambda)))\neq 0$. That is, there is a non-zero function $f\in\Sigma_{X}^{*}$ such that $D(f)+D+D(\lambda) \geq 0$, which means that
\[
	f\cdot\lambda  \in  U_{-D}. 
\]
Keeping in mind Lemma~\ref{lem:IX1top} and Theorem~\ref{thm:LocComIFFFinite}, one obtains that $U_{-D}\cap \I_{X}^{1}$ is compact, and hence so is $\I_{X}^{1}/\Sigma_{X}^{*}$.

Conversely,  assume that  $\I_{X}^{1}/\Sigma_{X}^{*}$ is compact. Let $D\in\Div(X)$ be a non-zero effective divisor. Since $V_{D}$ is an open subgroup, it is also closed. Since the map $V_{D}\to  \I_{X}^{1}/\Sigma_{X}^{*}$ is injective, it follows that $V_{D}$ is compact. Now, Theorem~\ref{thm:LocComIFFFinite} implies that $k$ is finite. 
\end{proof}

\begin{thm}
\label{T:equivalences}	
For a field $\Omega$ with $\Sigma_{X}\subset \Omega\subset \A_{X}$, the following conditions are equivalent:
\begin{enumerate}


\item\label{T:equivalences:item:dimfin}
$\dim_{k} \Omega\cap\A_{X}^+ <\infty$.

\item\label{T:equivalences:item:discrete}
$\Omega$ is discrete (and therefore closed). 

\item\label{T:equivalences:item:subsetIX1}
$\Omega^{*} \subset \I_{X}^{1}$.

\item\label{T:equivalences:item:discreteIX}
$\Omega^{*} $ is discrete in $\I_{X}$.

\setcounter{enumi}{4}
\item\label{T:equivalences:item:WcapAX+=k} 
$\Omega\cap\A_{X}^+ =k$.

\item\label{T:equivalences:item:W*capAX+=k*}  
$\Omega^{*}\cap \A_{X}^+ =k^{*}$.

\end{enumerate}
\end{thm}

\begin{proof}
\leavevmode

\begin{itemize}

\item $\eqref{T:equivalences:item:dimfin}\iff \eqref{T:equivalences:item:discrete}$: This follows from Proposition~\ref{P:discreteClosed}.

\item $\eqref{T:equivalences:item:dimfin}\implies \eqref{T:equivalences:item:subsetIX1}$:
Suppose that $\Omega^{*} \not\subset  \I_{X}^1$. Choose $\omega\in\Omega^{*}$ such that $\deg(\omega)>0$. There exists $f\in \Sigma_{X}^{*}$ and $n>0$ such that the divisor of $f\omega^n$ is effective, i.e. $f\omega^n \in \A_{X}^+$. Since $k[f\omega^n]\subseteq \Omega\cap \A_{X}^+$ is not finite-dimensional, e.g. because the powers $(f \omega^{n})^{i}$ for $i \geq 0$ are independent over $k$, we get a contradiction. 

\item $\eqref{T:equivalences:item:subsetIX1}\implies \eqref{T:equivalences:item:discreteIX}$:
Since $\{1\}$ is closed in $\Omega^{*}$, discreteness is equivalent to $\{1\}$ being open, i.e. that there exists an effective divisor $D$ such that $\Omega^{*}\cap V_{D}=\{1\}$. If this is not the case, then for all $D > 0$ there is an idele $\omega_{D}\in (\Omega^{*}\cap V_{D}) \setminus \{1\}$. Since $\Omega$ is a field, $1-\omega_{D} \in \Omega^{*}$. On the one hand, $1-\omega_{D}$ has zeros in the support of $D$, but on the other hand, it has no poles because $\omega_{D}\in V_{D}$. The hypothesis implies that $1-\omega_{D}=0$, which is a contradiction. 

\item $\eqref{T:equivalences:item:discreteIX}\implies \eqref{T:equivalences:item:discrete}$:
We know that $\eqref{T:equivalences:item:discreteIX}$ implies $\eqref{T:equivalences:item:subsetIX1}$, and the topology of $\I_{X}^1$ is induced from $\A_{X}$. It is easy to check that $\Omega \cap (1+U_{D}) = \Omega^{*} \cap V_{D}=\{1\}$, which shows that  $\{1\}$ is open in $\Omega\subset \A_{X}$. Thus $\Omega$ is discrete.

\item $\eqref{T:equivalences:item:WcapAX+=k}\implies \eqref{T:equivalences:item:W*capAX+=k*}$:
This is obvious since $\Omega^{*} \cap \A_{X}^+ \subset \Omega\cap \A_{X}^+ =k$ and $\Omega$ is a field.

\item $\eqref{T:equivalences:item:W*capAX+=k*}\implies \eqref{T:equivalences:item:WcapAX+=k}$:
Assume that $\omega\in \Omega\cap \A_{X}^+$  with $\omega\notin k$.  Since $\Omega$ is a field, we have $\omega\in\Omega^{*}$ and thus $\omega\in\Omega^{*}\cap \A_{X}^+=k^{*}$, a contradiction.

\item  $\eqref{T:equivalences:item:WcapAX+=k}\implies \eqref{T:equivalences:item:dimfin}$: Trivial.

%

\item $\eqref{T:equivalences:item:dimfin} \implies \eqref{T:equivalences:item:WcapAX+=k}$:
If $\dim_{k} \Omega\cap\A_{X}^+ <\infty$ then, given a point $x \in X$, for $n \gg 0$, $\Omega \cap \A_{X}^{+}$ embeds in $\widehat{\O}_{X,x}/\m_{x}^{n}$. 
We claim that therefore $\Omega \cap \A_{X}^{+}$ actually embeds in the residue field $k(x) = \widehat{\O}_{X,x}/\m_{x}$. Indeed, for any $\omega \in \Omega^{*} \cap \A_{X}^{+}$, we must have $v_{x}(\omega_{x}) = 0$, since otherwise the powers $\omega^{i}$ would be $k$-linearly independent. Therefore $\Omega^{*} \cap \A_{X}^{+}$ embeds in $k(x)^{*}$. Since this holds for any $x$, and $k \subseteq \Omega \cap \A_{X}^{+}$, the assumption that $k$ is algebraically closed in $\Sigma_{X}$ allows us to conclude that $\Omega \cap \A_{X}^{+} = k$.
\qedhere
\end{itemize}
\end{proof}

\section{Commutator pairings on ideles}
\label{subsec:pairings}

In this section, using the local symbol (\cite{Se}) and the Weil reciprocity law (\cite{Weil}), we associate to $X$ a topological subgroup of $\I_X^1/\Sigma_{X}^*$.  In  the forthcoming sections \S\ref{sec:TopPerp} and~\S\ref{sec:relation to the picard group}, it will be shown that this subgroup encodes arithmetic properties of the base field $k$ and also of the curve.

\subsection{Local and global symbols}
\label{subsec:local and global symbols}

\begin{defn}
\label{D:localSymbol}
Let $X$ be a smooth, complete and connected curve over a perfect base field $k$. For a closed point $x \in X$ with associated valuation $v_{x}$, completion $K_{x}$, and residue field $k(x)$, the \emph{local symbol at $x$} is given by
\begin{equation}
\label{E:tameSymbol}
	  \localsymbol{f,g}_{x}
	= (-1)^{\deg(x)\cdot v_{x}(f)\cdot v_{x}(g)}
	  \norm*{\frac{f^{v_{x}(g)}}{g^{v_{x}(f)}}(x)},
\end{equation}
where $f,g\in K_{x}^{*}$ and $\normSymbol:k(x)\to k$ denotes the norm map (applied to the residue classes).
\end{defn}

A few words regarding Definition~\ref{D:localSymbol} are in order. For $k$  algebraically closed, Serre in~\cite{Se} defines the local symbol as the pairing $\Sigma_{X}^{*}\times \Sigma_{X}^{*}  \longrightarrow k^{*}$ given by
\[
	\localsymbol{f,g}_{x} = (-1)^{v_{x}(f)v_{x}(g)} \frac{f^{v_{x}(g)}}{g^{v_{x}(f)}}(x).
\]
Keeping in mind that $\Sigma_{X}^{*} \hookrightarrow K_{x}^{*}$, both definitions coincide in this case.

\begin{defn}
\label{D:globalSymbol}
We define a global pairing
\[
	\globalsymbol{} : \I_{X} \times \I_{X} \to k^{*}
\]
by the following formula:
\begin{equation}
\label{eq:explicitPairI1}
    \begin{aligned}
		  \globalsymbol{\alpha,\beta}
		&= \prod_{x \in X} \localsymbol{\alpha_{x},\beta_{x}}_{x}
		 = \prod_{x\in X} (-1)^{\deg(x)\cdot v_{x}(\alpha_{x}) \cdot v_{x}(\beta_{x})}
		   \norm[\bigg]{\frac {\alpha^{v_{x}(\beta_{x})}}{\beta^{v_{x}(\alpha_{x})}}(x)},
	\end{aligned}
\end{equation}
where for an idele $\lambda$ and a closed point $x$, $\lambda(x)$ is the residue class of $\lambda_{x}$ in $k(x)$.
\end{defn}

Both the definitions of the local and global pairings which we give here are in fact separate constructions (see~\cite{MP,Pa} for details) derived from commutator pairings associated to certain central extensions. Thus formula ~\eqref{eq:explicitPairI1} is in fact a theorem, not really a definition. In any case, in view of these facts, the global pairing satisfies 
\begin{equation}
\label{E:propsSymbol}
\begin{aligned}
	   \globalsymbol{\alpha_{1} \alpha_{2},\beta}
	&= \globalsymbol{\alpha_{1}, \beta} \cdot \globalsymbol{\alpha_{2}, \beta},
	\\
	   \globalsymbol{\alpha, \beta_{1} \beta_{2}}
	&= \globalsymbol{\alpha, \beta_{1}} \cdot \globalsymbol{\alpha, \beta_{2}},
	\\
	   \globalsymbol{\alpha,\beta} \cdot \globalsymbol{\beta,\alpha}
	&= 1,
	\\
	   \globalsymbol{\alpha,-\alpha} 
	&= 1,
	\\
	   \globalsymbol{\alpha,\alpha} 
	&= \globalsymbol{\alpha,-1},
	\\
	   \globalsymbol{\alpha,1-\alpha} 
	&= 1,
	\quad\text{if $\alpha, 1-\alpha \in \I_{X}$,}
	\\
	   \globalsymbol{\alpha,\alpha} 
	&= 1, 
	\quad\text{if $\alpha \in \I_{X}^{1}$,}
\end{aligned}
\end{equation}
which is to say, it is a symbol in the sense of algebraic $K$-theory (see for instance Tate~\cite{TaSym}).
Moreover, the Weil reciprocity law corresponds to the statement that
\begin{equation}
\label{E:WeilRec}
	\globalsymbol{f,g}  = \prod_{x\in X} \localsymbol{f,g}_{x} = 1 \text{ for all }
	f,g \in \Sigma_{X}^{*},
\end{equation}
where $\Sigma_{X}^{*}$ is embedded diagonally in $\I_{X}$, i.e. triviality of the global pairing on $\Sigma_{X}^{*}$.

\begin{rem}
In a similar way as Tate~(\cite{Ta}) deduced the residue theorem from the properties of a cocycle, we now rely on the properties of local symbols to obtain our results. The consequences of local symbols have been widely studied (\cite{AP,ACK}) and, among them, it note the proof of the Weil reciprocity law given in \cite{MP}. Indeed, the ideas of  \cite{MP} together with the techniques of \cite{Colle} might be used to study analogues for arithmetic curves. Furthermore,  Artin and Whaples (\cite{ArtinW}) characterized certain fields in terms of a \textit{product formula} for valuations; however, they need archimedean valuations in the case of function fields, while we consider non-archimedean valuations. Finally, the relation of our results to the classical approach of class field theory in terms of $K$-theory (e.g. \cite{TaSym}) deserves further research. 
\end{rem}

\begin{prop}
\label{P:locConst}
Each $(\alpha_{0},\beta_{0}) \in \I_{X} \times \I_{X}$, has a clopen neighborhood on which the global pairing $\globalsymbol{}$ is constant.
\end{prop}

\begin{proof}
We begin by showing that, given $\beta_{0} \in \I_{X}$, there is a neighborhood $V_{D}$ of $1$ of the form~\eqref{E:basisIdelesVD} such that $\globalsymbol{\alpha,\beta_{0}} = 1$ for all $\alpha \in V_{D}$. Indeed, let $D(\beta_{0})$ be the divisor associated to $\beta_{0}$ by~\eqref{eq:IXaDivisor} and choose an effective divisor $D \geq D(\beta_{0})$. Since $V_{D} \subseteq V_{0}$, for $\alpha \in V_{D}$ we have $v_{x}(\alpha_{x}) = 0$ at all $x$, thus
\[
	  \localsymbol{\alpha,\beta_{0}}_{x}
	= \norm[\big]{\alpha(x)^{v_{x}(\beta_{0,x})}}.
\]
If $x \in \supp(D)$ then $\alpha_{x} \in 1 + \m_{x}$ and if $x \notin \supp(D)$ then also $x \notin \supp D(\beta_{0})$ and so $v_{x}(\beta_{0,x}) = 0$. In either case the above norm is $1$, hence $\globalsymbol{\alpha,\beta_{0}} = \prod_{x} 1 = 1$.

In general, the multiplicativity of the global pairing implies that
\[
	  \globalsymbol{\alpha_{0}\gamma,\beta_{0}\delta}
	= \globalsymbol{\alpha_{0},\beta_{0}} \globalsymbol{\alpha_{0},\delta}
	  \globalsymbol{\gamma,\beta_{0}} \globalsymbol{\gamma,\delta},
\]
and thus $\globalsymbol{\alpha,\beta} = \globalsymbol{\alpha_{0},\beta_{0}}$ for $(\alpha,\beta) \in \alpha_{0}V_{D} \times \beta_{0}V_{D}$ where $D \geq 0,D(\alpha_{0}), D(\beta_{0})$.
\end{proof}

\subsection{Orthogonal complements}
\label{subsec:orthogonal complements}

In what follows, we will focus on the global pairing $\globalsymbol{}$ and in particular on its restriction to $\I_{X}^{1}$. We will study the various related notions of orthogonality that arise with respect to this pairing, beginning with the radicals. Consider the radical of the pairing $\globalsymbol{}$ on $\I_{X}^{1}$,
\[
	R^{1} := \big\{\alpha \in \I_{X}^{1} \:\vert\: \globalsymbol{\alpha,\beta} = 1 \  \forall\beta\in \I_{X}^{1}\big\}.
\]

\begin{thm}
\label{T:rad1}
\leavevmode{}
\begin{itemize}

\item $k^{*}\cdot \prod_{x\in X}\big((1+\m_{x})N_{x}\big) = R^{1} \cap V_{0}$, where $N_{x}:=\{ \lambda\in k(x) \:\vert\: \norm{\lambda}=1\}$. 

\item If $k$ is infinite, then $R^{1} \subseteq V_{0}$.

\item $k$ is algebraically closed if and only if  $R^{1} =k^{*}\prod_{x\in X}(1+\m_{x})$,

\end{itemize}

\end{thm}

\begin{proof}
It is  straightforward to verify the inclusion $k^{*}\cdot \prod_{x\in X}\big((1+\m_{x})N_{x}\big) \subseteq R^{1}$ directly from the definitions and the fundamental relation~\eqref{eq:explicitPairI1}. For instance, for $c \in k^{*}$ we have
\[
	  \globalsymbol{c,\beta}
	= c^{\deg D(\beta)}
	= c^{\sum_{x} \deg(x) v_{x}(\beta_{x})}
\]
and the exponent is null precisely when $\beta \in \I_{X}^{1}$.

Let us now see that existence of a rational point $x_{0}$ yields equality.
Since $k(x_{0})=k$ and we have already observed that $k^{*}$ is contained in the
radical, we may assume that $\alpha_{x_{0}}= u_{x_{0}}\in 1+\m_{x_{0}}$. Now,
choose an arbitrary point $x_{1}$ and write $\alpha_{x_{1}} = c_{1} \cdot u_{x_{1}}$
with $c_{1} \in k(x_{1})^{*}$. It suffices to show that $N_{k(x_{1})/k}(c_{1})=1$. Let
us define
\[
    \beta=(\beta_{x})=\begin{cases}
    t_{x_{0}}^{\deg(x_{1})} & \text{ if } x=x_{0},
	                        \\
    t_{x_{1}}^{-1}          & \text { if } x=x_{1},
	                        \\
    1                       & \text{ if } x\not=x_{0},x_{1}.
    \end{cases}
\]
Recalling~\eqref{eq:IXaDivisor}, note that $\deg(D(\beta))=\deg(\deg(x_{1})x_{0} - x_{1})= \deg(x_{1})-\deg(x_{1})=0$, and thus $\beta\in\I_{X}^{1}$. Now $1= \globalsymbol{\alpha,\beta}=N_{k(x_{1})/k}(c_{1})^{-1}$ and the result follows.

We now check that when $k$ is infinite, $R^{1} \subseteq V_{0}$, i.e. for $\alpha \in R^{1}$, $v_{x}(\alpha_{x})=0$ for all $x\in X$. Suppose to the contrary that $\alpha \in R^{1}$ but $v_{x}(\alpha_{x}) \neq 0$ for some $x$. Write $\alpha_{x}=c_{x} \cdot u_{x} \cdot t_{x}^{m_{x}}$, where $t_{x}$ is a local parameter at $x\in X$, $c_{x} \in k(x)^{*} \hookrightarrow \widehat{\O}_{X,x}^{*}$, $u_{x}\in 1+\m_{x}$ and $m_{x} = v_{x}(\alpha_{x}) \in \Z$. Since $k$ is infinite, there is certainly some $\lambda \in k(x)^{*}$ (indeed, in $k^{*}$ itself) with $\norm{\lambda}^{m_{x}} \neq 1$. Consider $\beta \in \I_{X}^{1}$ given by
\[
\beta=(\beta_{y})=\begin{cases}
        \lambda &\text{ if } y=x,
		        \\
        1       &\text{ if } y\not=x.
        \end{cases}
\]
Then $\globalsymbol{\alpha,\beta} = \norm{\lambda}^{-m_{x}} \neq 1$, which contradicts our assumption. Hence $v_{x}(\alpha_{x})=0$ for all $x\in X$.

It remains to show that $R^{1} =k^{*}\prod_{x\in X}(1+\m_{x})$ if and only if $k$ is algebraically closed. If $k$ is algebraically closed, the norm one groups $N_{x}$ are trivial. Conversely, if the equality holds, then $N_{x}$ is trivial for all (closed) $x \in X$. This implies that $k(x) = k$ for all $x$ (see~\cite{Ono} for a general characterization of the kernel of the norm map in a separable extension), i.e., every closed point is rational, and hence $k$ must be algebraically closed.
\end{proof}

Consider now the radical of the pairing $\globalsymbol{}$ on $\I_{X}$,
\[
R := \big\{\alpha \in \I_{X} \:\vert\: \globalsymbol{\alpha,\beta} = 1 \  \forall\beta\in \I_{X}\big\}. 
\]

\begin{thm}
\label{T:rad}\leavevmode
\begin{itemize}
	
	\item $ \prod_{x\in X}\big((1+\m_{x})N_{x}\big) = R \cap V_{0}$.
	
	\item If $k$ is infinite, then $R \subseteq V_{0}$.
	
	\item $k$ is algebraically closed if and only if  $R =\prod_{x\in X}(1+\m_{x})$,
	
\end{itemize}

\end{thm}

\begin{proof}
The proof is analogous to that of Theorem~\ref{T:rad1}, except that the factor $k^{*}$ is now absent.
\end{proof}

\begin{rem}
If $k$ is a finite field, then the radical $R$ may contain elements of non-zero degree.  
\end{rem}

\begin{defn}
For a subset $G\subset {\I}_{X}^{1}$, we define the ``orthogonal complement'' of $G$ in $\I_{X}^{1}$ by
\[
    G^{\perp}  := 
    \{ \alpha\in {\I}_{X}^{1} \:\vert\: \globalsymbol{\alpha,g}=1 \ \forall g\in G\}.
\]
\end{defn}

Since $\prod(1+{\m}_{x})$ is contained in the radical $R$ of the pairing $\globalsymbol{}$, there is an induced pairing on the quotient $\bar{\I}_{X}^{1}= \I_{X}^{1} /\prod(1+{\m}_{x})$, which we will denote in the same way, and with respect to which we may also study orthogonal complements. Note that the induced pairing is also locally constant by Proposition~\ref{P:locConst}.

\begin{defn}
For $G\subseteq \bar{\I}_{X}^{1}$, we define its orthogonal complement in $\bar{\I}_{X}^{1}$ by
\[
	G^{\pperp}:= 
	\{ \alpha\in \bar{\I}_{X}^{1} \:\vert\: \globalsymbol{\alpha,g}=1 \ \forall g\in G\}.
\]
\end{defn}

\begin{prop}
\label{P:propsPerp}
\leavevmode
\begin{enumerate}

\item For any subset $G \subseteq \I_{X}^{1}$, the orthogonal complement $G^{\perp}$ is a closed subgroup of $\I_{X}^{1}$ containing the radical $R^{1}$.

\item $\I_{X}^{1} \subseteq (k^{*})^{\perp}$, with equality when $k$ is infinite.
In the latter case, if $k^{*} \subseteq G \subseteq \I_{X}$ and $\alpha \in \I_{X}$ satisfies $\globalsymbol{\alpha,g}=1 \ \forall g\in G$, then actually $\alpha \in \I_{X}^{1}$.

\item If $p:\I_{X}^{1}\to \bar{\I}_{X}^{1}$ is the projection map, then for any subset $G \subseteq \bar{\I}_{X}^{1}$ we have $p^{-1}(G^{\pperp}) = p^{-1}(G)^{\perp}$, and for any subset $G \subseteq \I_{X}^{1}$, we have $p(G)^{\pperp} = p(G^{\perp})$.

\item For any subset $G \subseteq \bar{\I}_{X}^{1}$, the orthogonal complement $G^{\pperp}$ is a closed and totally disconnected subgroup of $\bar{\I}_{X}^{1}$.

\end{enumerate}
\end{prop}

\begin{proof}
\leavevmode
\begin{enumerate}

\item That $G^{\perp}$ is a subgroup follows from the fact that $1 \in R^{1}$ and the pairing is an alternating form on $\I_{X}^{1}$. That it is closed is a consequence of Proposition~\ref{P:locConst}. Clearly $R^{1} = (\I_{X}^{1})^{\perp} \subseteq G^{\perp}$.

\item Simply note that $\globalsymbol{\alpha,c} = c^{-\deg D(\alpha)}$ for all $c \in k^{*}$.

\item Since $p$ is surjective, $p^{-1}(G^{\pperp}) = p^{-1}(G)^{\perp}$. From this we obtain $p^{-1}(p(G)^{\pperp}) = (p^{-1}(p(G)))^{\perp} = (G \cdot \prod_{x} (1 + \m_{x}))^{\perp} = G^{\perp}$, therefore $p(G)^{\pperp} = p(G^{\perp})$.

\item 

Since $\bar{\I}_{X}^{1}$ is totally disconnected (Theorem~\ref{thm:barI-T2totdis}) and $G^{\pperp}$ is a subgroup, $G^{\pperp}$ is also totally disconnected. It is closed because the pairing on $\bar{\I}_{X}^{1}$ is also locally constant.
\qedhere
\end{enumerate}
\end{proof}

\section{\texorpdfstring{Topological properties of $(\Sigma_{X}^{*})^{\pperp}/ \Sigma_{X}^{*}$}{Topology of the quotient}}\label{sec:TopPerp}

This section is mainly devoted to exhibiting how $\I_{X}^{1}/\Sigma_{X}^{*}$ encodes some arithmetical properties of the base field. Indeed, it is well-known in class field theory  that $\I_{X}^{1}/\Sigma_{X}^{*}$ is profinite for $k$ finite (\cite{CaFr}), but now we can also prove the converse implication (Proposition~\ref{p:IX1/SigmaProfinite}, see also Theorems~\ref{thm:LocComIFFFinite},~\ref{thm:PropertiesSigma*} for related results). This section finishes with the case when $X={\mathbb P}_1$, although the discussion of the dependence on the curve will be carried out in the next section,~\S\ref{sec:relation to the picard group}. 

In what follows, we continue to let $p:\I_{X}^{1}\to \bar{\I}_{X}^{1} = \I_{X}^{1} / (\prod_{x\in X}(1+{\m}_{x}))$ denote the quotient map. 

Note that $\Sigma_{X}^{*}\cap \prod(1+{\m}_{x})=\{1\}$, and thus, with another small but convenient abuse of notation, we may write $\Sigma_{X}^{*}\subseteq \bar{\I}_{X}^{1}$. Moreover, the triviality of the global commutator pairing~\eqref{E:WeilRec} on $\Sigma_{X}^{*}$ (which is equivalent to the Weil reciprocity law) implies that
\[
	\Sigma_{X}^{*}\subseteq (\Sigma_{X}^{*})^{\pperp} \subseteq 
	\bar {\I}_{X}^{1}.
\]
It follows from Proposition~\ref{P:propsPerp} that
\begin{equation}\label{E:SigmaPPerp}
	(\Sigma_{X}^{*})^{\pperp} = p((\Sigma_{X}^{*})^{\perp}) \simeq  (\Sigma_{X}^{*})^{\perp}/ \prod_{x\in X}(1+{\m}_{x}). 
\end{equation}

\begin{prop}
\label{prop:PropertiesSigmaOrt}
$(\Sigma_{X}^{*})^{\pperp}$ is closed, Hausdorff, totally disconnected and with empty interior in $\bar{\I}_{X}^{1}$.
\end{prop}


\begin{proof}
The fact that $(\Sigma_{X}^{*})^{\pperp}$ is Hausdorff and totally disconnected follows from Theorem~\ref{thm:barI-T2totdis}, and it is closed by Proposition~\ref{P:propsPerp}.
Let us then show that it has empty interior. Since $(\Sigma_{X}^{*})^{\perp}=p^{-1}((\Sigma_{X}^{*})^{\pperp})$, it suffices to show that  $(\Sigma_{X}^{*})^{\perp}$ has empty interior. Otherwise, there exists $\lambda\in (\Sigma_{X}^{*})^{\perp}$ and an effective divisor $D$ such that $\lambda V_{D} \subseteq (\Sigma_{X}^{*})^{\perp}$. Multiplying by $\lambda^{-1}$, it follows that $V_{D} \subseteq (\Sigma_{X}^{*})^{\perp}$. Choose a closed point $x_{0} \in X \setminus \supp (D)$ such that $k(x_{0})$ has more than two elements and choose an idele $\mu\in\I_{X}^{1}$ with $v_{x}(\mu_{x})=0$ for all $x$, $\mu_{x}=1$ for all $x \neq x_{0}$ and $\mu(x_{0}) \neq 1$. Then $\mu\in V_{D}$, and an easy computation yields
\[
	\globalsymbol{\mu,f}= \mu(x_{0})^{v_{x_{0}}(f)},
\]
which has to be equal to $1$ for all $f\in\Sigma_{X}^{*}$ since $V_{D} \subseteq (\Sigma_{X}^{*})^{\perp}$. This implies that $\mu(x_{0})=1$, which contradicts our choice of $\mu$.

\end{proof}

Note that, contrary to what one might suspect in light of Proposition~\ref{prop:PropertiesSigmaOrt}, $(\Sigma_{X}^{*})^{\pperp}$ is not necessarily discrete in $\bar{\I}_{X}^{1}$. See Proposition~\ref{prop:normInQuotient} for more on this.

\begin{lem}
\label{lem:closed}
$\Sigma_{X}^{*}$ is closed in $\bar{\I}_{X}^{1}$.
\end{lem}

\begin{proof}
It suffices to show that given $\bar\lambda \in \bar{\I}^1_{X}$ with $\bar\lambda\notin p(\Sigma_{X}^{*})$, there exists a finite subset $S$ such that:
\[
	p(\Sigma_{X}^{*})  \cap \bar\lambda  \bar{V}_{S} = \emptyset ,
\]
for which it is enough to check that their preimages do not intersect:
\[
	\Bigl(
		\Sigma_{X}^{*} \prod_{x\in X}(1+{\m}_{x})
	\Bigr)
	\:\bigcap\:
	\lambda
	\Big(
		\prod_{x\in S}(1+{\m}_{x})
		\times
		\prod_{x\notin S}\widehat{\O}_{X,x}^{*}
	\Big)
	= \emptyset.
\]
Since the subgroup $\prod_{x\in X}(1+{\m}_{x})$ acts on both terms in the intersection, this is equivalent to showing that there exists a finite $S$ with
\[
	\Sigma_{X}^{*}  
	\:\bigcap\:
	\lambda
	\Big(
		\prod_{x\in S}(1+{\m}_{x})
		\times
		\prod_{x\notin S}\widehat{\O}_{X,x}^{*}
	\Big)
	= \emptyset.
\]
Suppose, to the contrary, that the intersection is not empty for all $S$, and for each finite subset $S\subseteq X$, choose a function $f_{S}\in\Sigma_{X}^{*}$ lying in the intersection. Then
\begin{equation}
\label{eq:f_S}
\begin{aligned}
	v_{x}(f_{S}) &= v_{x}(\lambda) &\quad &\forall x\in X ,
	\\
	\frac{f_{S}}{\lambda}(x) &= 1 &\quad &\forall x\in S.
\end{aligned}
\end{equation}

Suppose that $S,S' \subseteq X$ are finite subsets with $S\cap S'\neq \emptyset$ and $S\neq S'$. Then $v_{x}(f_{S})=v_{x}(f_{S'})$ for all $x\in X$. Since $X$ is complete and connected, there exists a point $a$ in the algebraic closure of $k$ satisfying $f_{S}= a f_{S'}$. Evaluating at any point of $S\cap S'$, we get $a=1$. This shows that $f_{S}$ does not depend on $S$, and we can simply denote this function by $f$. Then  \eqref{eq:f_S} implies that in fact $f = \bar\lambda$ in $\bar\I_{X}^{1}= \I_{X}^{1}/ \prod_{x\in X}(1+{\m}_{x})$, and consequently, that $\bar\lambda \in p(\Sigma_{X}^{*})$, which contradicts the choice of $\lambda$.
\end{proof}

\begin{thm}
\label{thm:SigmaPerpHausTotDis}
$(\Sigma_{X}^{*})^{\pperp}/\Sigma_{X}^{*} \subset \bar\I_{X}^{1}/\Sigma_{X}^{*} $ is closed and both groups are Hausdorff and totally disconnected. 
\end{thm}

\begin{proof} 
Proposition~\ref{prop:PropertiesSigmaOrt} implies that $(\Sigma_{X}^{*})^{\pperp}/\Sigma_{X}^{*} \subset \bar\I_{X}^{1}/\Sigma_{X}^{*} $ is closed. By Lemma~\ref{lem:closed}, the quotient space $\bar\I_{X}^{1}/\Sigma_{X}^{*} $ is Hausdorff. Since $(\Sigma_{X}^{*})^{\pperp}/\Sigma_{X}^{*}$ is a subspace of a Hausdorff space, it is also Hausdorff. 

Consider the quotient map $\bar{p}: \I_{X}^{1}\to \bar\I_{X}^{1}/\Sigma_{X}^{*} $. It is straightforward to check that the subsets $\Sigma_{X}^{*} V_{D} $ are open subgroups of $\I_{X}^{1}$ and, since $\Sigma_{X}^{*} V_{D}=\bar{p}^{-1}(\bar{p}( \Sigma_{X}^{*} V_{D}))$, that the collection
\[
	\big\{ \bar{p}(\Sigma_{X}^{*} V_{D}) \:\vert\: D \in \Div(X), \text{ $D$ is  effective}\big\}
\]
is a basis of neighborhoods of $1$ in $\bar\I_{X}^{1}/\Sigma_{X}^{*} $. Since this basis consists of open subgroups of $\bar\I_{X}^{1}/\Sigma_{X}^{*} $, it follows that $\bar\I_{X}^{1}/\Sigma_{X}^{*} $ is totally disconnected, and since subspaces of totally disconnected spaces are also totally disconnected, the conclusion follows. 
\end{proof}

\begin{prop}\label{p:IX1/SigmaProfinite}	
$\bar\I_{X}^{1}/\Sigma_{X}^{*}$ is profinite if and only if $k$ is finite.
\end{prop}

\begin{proof}
A topological group is profinite if and only if it is Hausdorff, compact and totally disconnected. Keeping in mind Theorem~\ref{thm:SigmaPerpHausTotDis}, this follows from Theorem~\ref{thm:PropertiesSigma*}.
\end{proof}

\begin{cor}
\label{cor:profinite}	
If $k$ is finite, then $(\Sigma_{X}^{*})^{\pperp}/\Sigma_{X}^{*}$ is profinite.
\end{cor}

\begin{proof}
Clear from Theorem~\ref{thm:SigmaPerpHausTotDis} and Proposition~\ref{p:IX1/SigmaProfinite}.
\end{proof}

\begin{prop}
\label{prop:normInQuotient}
Consider the subgroup $\prod_{x\in X} N_{x} \subset \bar\I_{X}^{1}$. Then
\begin{equation}\label{E:NormsPerp}
	\prod_{x\in X} N_{x}  \subseteq  (\Sigma_{X}^{*})^{\pperp} / \Sigma_{X}^{*}.
\end{equation}
In addition, if $(\Sigma_{X}^{*})^{\pperp} / \Sigma_{X}^{*}$ is discrete, then $k$ must be algebraically closed.
\end{prop}

\begin{proof}
The inclusion follows from Theorem~\ref{T:rad1}. 

If $k$ is not algebraically closed, then $X$ has infinitely many closed points (i.e. defined over $\bar k$) which are not rational (i.e. not defined over $k$) and, since $k$ is assumed perfect, $k\hookrightarrow k(x)$ is separable. For such a point $x$ one necessarily has $N_x\neq \{1\}$ (see for example~\cite{Ono}).

It is thus clear that $\big(\prod_{x\in X} N_{x}\big) \cap \bar V_S \neq \{\bar 1\}$ in $\bar\I_X^1$ for any finite subset $S\subset X$. Furthermore, one checks that  
	\[
	p\big(\prod_{x\in X} N_{x}\big) \cap p(\bar V_S) \, \neq\,  \{\bar 1\} \qquad \forall S
	\] 
where $p: \bar\I_X^1 \to\bar\I_X^1 / \Sigma_{X}^{*}$ is the projection, and thus $p\big(\prod_{x\in X} N_{x}\big)$ cannot be discrete. Thus by~\eqref{E:NormsPerp}, one gets a contradiction if $(\Sigma_{X}^{*})^{\pperp} / \Sigma_{X}^{*}$ is discrete.

\end{proof}

\begin{cor}
\label{cor:algcerrP1}
Let $X = \mathbb{P}_{1}$. If $k$ is infinite, then
\[
	\prod_{x\in X} N_{x} = (\Sigma_{X}^{*})^{\pperp} / \Sigma_{X}^{*},
\]
and hence
\[
	(\Sigma_{X}^{*})^{\pperp}/\Sigma_{X}^{*} = 0
\]
if and only if $k$ is algebraically closed.
\end{cor}

\begin{proof}
Suppose $X={\mathbb{P}_{1}}$ and let $\alpha\in  (\Sigma_{X}^{*})^{\pperp}$. Since $X = {\mathbb{P}}_{1}$, there is a rational function $f\in\Sigma_{X}^{*}$ such that $D(\alpha)=D(f)$. It is enough to show that $f^{-1}\cdot  \alpha \in  \prod_{x\in X} N_{x} $. Let $x_{0}$ be a rational point and let $x_{1}$ be an arbitrary closed point. Consider a rational function $\beta\in \Sigma_{X}^{*}$ with divisor $\deg(x_{1}) x_{0}- x_{1}$. The result follows arguing as in the proof of Theorem~\ref{T:rad1}. 	
The rest is clear.
\end{proof}


\section{\texorpdfstring{Relation of $(\Sigma_{X}^{*})^{\pperp}/\Sigma_{X}^{*}$ to the Picard group}{Relation to the Picard group}}
\label{sec:relation to the picard group}

This last section shows how the quotient $(\Sigma_{X}^{*})^{\pperp}/\Sigma_{X}^{*}$ relates to the Picard group of $X$, and thus unveils how this subgroup depends on the curve. Besides some general results (Theorem~\ref{thm:EXACTSEQSigma}), the cases of $k=\C$ and $X$ an elliptic curve over $\C$ will be worked out in detail. Finally, we deduce further consequences which apply to field extensions of $\Sigma_X$ (see Theorem~\ref{T:SigmaAlg}).

To begin with, let us observe that the map \eqref{eq:IXaDivisor} induces a map
\[
	(\Sigma_{X}^{*})^{\pperp}/\Sigma_{X}^{*} \longrightarrow   
\Pic^0(X),
\]
which constitutes the main object of study in this section.

A couple of remarks are in order. On the one hand, note that the global pairing \eqref{eq:explicitPairI1} yields a map
\[
\begin{aligned}
(\Sigma_{X}^*)^{\perp} & \longrightarrow \operatorname{Hom}_{\text{gr}}(\I_X^1, k^*) 
\\
\alpha & \longmapsto \quad \beta \mapsto \langle \alpha, \beta \rangle
\end{aligned}
\]
which, by Weil reciprocity~(\eqref{E:WeilRec} and \eqref{E:SigmaPPerp}), factors as
\[
\begin{aligned}
\Phi\,\colon\, (\Sigma_{X}^*)^{\pperp} & \longrightarrow \operatorname{Hom}_{\text{gr}}(\bar\I_X^1/\Sigma_{X}^*, k^*) 
\\
\alpha & \longmapsto \quad \beta \mapsto \langle \alpha, \beta \rangle.
\end{aligned}
\]
On the other hand, the global pairing \eqref{eq:explicitPairI1} also induces a map
	\begin{equation}\label{E:PairingPic}
\begin{aligned}
\Div^0(X) \times \bar V_{0} & \longrightarrow \, k^* \\
\big( \sum_{x \in X} n_x x \, , \alpha \big) & \longmapsto \, \prod_{x \in X} N_{k(x)/k}(\alpha(x))^{n_x}
\end{aligned}
\end{equation}
(recall the definition of the open neighborhood $\bar V_{0}\subset \bar\I_X^1$ from Theorem~\ref{thm:barI-T2totdis}).

Before stating the results, we make a small digression.  Observe that the sequence
\[
	0\to \bar V_{0}/k^* \to \bar\I_X^1/\Sigma_X^* \to \bar\I_X^1/\bar V_{0}\cdot \Sigma_X^* \to 0 
\]
is exact. We will use the following well-known identifications
\[
	\I_X^1 /  V_{0} \simeq \bar\I_X^1 /  \bar V_{0} \simeq \Div^0(X) 
	\, , \qquad 
	\I_X^1 / \Sigma_{X}^* V_{0} \simeq \bar\I_X^1 /  \Sigma_{X}^* \bar V_{0} \simeq \Pic^0(X)\, . 
\]
Note that $(\Sigma_{X}^{*})^{\perp} \cap \A_{X}^+ =
(\Sigma_{X}^{*})^{\perp}\cap  V_{0} =  \ker\big((\Sigma_{X}^{*})^{\perp}\to \Div^0(X)\big)$, where the morphism $(\Sigma_{X}^{*})^{\perp}\to \Div^0(X)$ is the restriction of \eqref{eq:IXaDivisor}.
It is now straightforward to check that the following diagram is commutative:
\[
\xymatrix@C=14pt{
	0 \ar[r] &  	(\Sigma_{X}^{*})^{\perp} \cap \A_{X}^+  \ar[r] \ar@{-->}[d] &   
	(\Sigma_{X}^{*})^{\perp} 
	\ar[r]  \ar[d]^{\Phi} &   
	 \Div^0(X) \ar[d]^{\Psi}
	\\
	0  \ar[r] &   \hom_{\text{gr}}\big(\Pic^0(X) ,k^{*}\big)  \ar[r] & 
	\hom_{\text{gr}}\big( \bar\I_X^1/ \Sigma_X^* ,k^{*}\big)  \ar[r] &
	\hom_{\text{gr}}\big(\bar V_{0}/ k^* ,k^{*}\big),  
}\]
where $\Psi$ is induced by \eqref{E:PairingPic} and the dashed arrow on the left-hand side is induced by $\Phi$. The map we want to study is related to this diagram as follows:
\[
\xymatrix{
	(\Sigma_{X}^{*})^{\perp} 
\ar[r]  \ar[d] &   
\Div^0(X) \ar[d]
\\
	(\Sigma_{X}^{*})^{\pperp} / \Sigma_{X}^{*}   
\ar[r]  &   \Pic^0(X).
}
\]
The notation $\hom_{\text{gr}}$ denotes homomorphisms of groups, ignoring any previous topological structure (equivalently, endowing all groups with the discrete topology).

\begin{prop}
\label{prop:ExSeqOrto}
There are maps $\iota,\pi$ such that the following sequence is exact:
\[
	0 \to \prod_{x\in X}(1+\m_{x}) N_{x}  \overset{\iota}\to  
	(\Sigma_{X}^{*})^{\perp} \cap \A_{X}^+ \overset{\pi}{\to}
	\hom_{\mathrm{gr}}\big(\Pic(X),k^{*}\big). 
\]
Furthermore, if $k$ is either finite or algebraically closed, the map $\pi$ is surjective. 
\end{prop}

\begin{proof}
 Recalling the notation~\eqref{E:basisIdelesVD} and Theorem~\ref{T:rad}, we have $\prod_{x\in X}(1+\m_{x}) N_{x} \subseteq   \A_{X}^+ \cap (\Sigma_{X}^{*})^{\perp} = \ker\big((\Sigma_{X}^{*})^{\perp}\to \Div(X)\big)$, so $\iota$ is the inclusion.

$\pi$ is defined as follows. Given an idele $\lambda$, we consider the following map from $\Div(X)$ to $k^{*}$:
\begin{equation}
\label{E:DivTok}
	D=\sum_{x \in X} n_{x} x \longmapsto \lambda(D):=\prod_{x\in X} \norm{\lambda(x)}^{n_{x}}.	
\end{equation}
We check that $\lambda(D(f))=1$ for $\lambda\in \ker\big((\Sigma_{X}^{*})^{\perp}\to \Div(X)\big)$ and $f\in\Sigma_{X}^{*}$. Indeed, since $v_{x}(\lambda_{x})=0$ for all $x$, one has
\[
	\begin{split}
	   1
	&= \globalsymbol{\lambda,f}
	 = \prod_{x} (-1)^{\deg(x)v_{x}(\lambda_{x})v_{x}(f)}    
	   \norm*{\frac{\lambda^{v_{x}(f)}}{f^{v_{x}(\lambda_{x})}}(x)}
	\\
	&= \prod_{x\in X} \norm{\lambda(x)}^{v_{x}(f)}
	 = \lambda(D(f)),
	\end{split}
\] 	
and accordingly, the morphism $\lambda:\Div(X)\to k^{*}$ induces a map $\Pic(X):=\Div(X)/\Sigma_{X}^{*}\to k^{*}$. This map is $\pi(\lambda)$. 

Let us now prove the exactness. Clearly, $\iota$ is injective. To check that $\pi(\iota(\lambda))=1$, it suffices to observe that for $\lambda\in \prod_{x\in X}(1+\m_{x}) N_{x}$, we have $N_{k(x)/x}(\lambda(x)) = 1$ for all $x \in X$.
It remains to check that $\ker\pi\subseteq \im\iota$. Let $\lambda\in\ker(\pi)$.  First, $\lambda\in \ker\big((\Sigma_{X}^{*})^{\perp}\to \Div(X)\big) \subseteq V_{0}$, therefore $v_{x}(\lambda_{x})=0$ for all $x\in X$. On the other hand, for the divisor $D_{x}:=x \in\Div(X)$, the fact that $\lambda\in\ker(\pi)$ yields
\[
	1= \lambda(D_{x})  = \norm{\lambda(x)}.
\]
Now $v_{x}(\lambda_{x})=0$ and $\lambda(x) \in N_{x}$ imply that $\lambda\in \prod_{x\in X}(1+\m_{x}) N_{x}$.

Finally, we show that $\pi$ is surjective if $k$ is either finite or algebraically closed. To a given group homomorphism $\phi:\Pic(X)\to k^{*}$, we associate the idele $\lambda^{\phi}$ defined by:
\begin{itemize}
	
	\item $v_{x}(\lambda^{\phi})=0$ for all $x$;
	
	\item $\lambda^{\phi}_{x} \in \widehat{\O}_{X,x}^{*}$, such that $\norm{\lambda^{\phi}(x)} = \phi(D_{x})$.
	
\end{itemize}
The existence of such a $\lambda^{\phi}$ is guaranteed either when $k$ is finite (by surjectivity of the norm) or when $k$ is algebraically closed.

We need to show that $\lambda^{\phi}\in \ker\big((\Sigma_{X}^{*})^{\perp}\to \Div(X)\big)$. Trivially, $D(\lambda^{\phi})=0$, so it suffices to verify that $\lambda^{\phi}\in (\Sigma_{X}^{*})^{\perp}$. For $f\in\Sigma_{X}^{*}$, an easy computation yields
\begin{align*}
	   \globalsymbol{\lambda^{\phi},f}
	&= \prod_{x} (-1)^{\deg(x)v_{x}(\lambda^{\phi})v_{x}(f)}
	   \norm*{\frac{(\lambda^{\phi})^{v_{x}(f)}}{f^{v_{x}(\lambda^{\phi})}}(x)}
	\\
	&= \prod_{x\in X} \norm{\lambda^{\phi}(x)}^{v_{x}(f)}
	 = \phi\big(\sum_{x\in X} v_{x}(f) x\big) = \phi(D(f)) = 1.
	\qedhere
\end{align*}
\end{proof}

\begin{cor}
\label{cor:KerSigmaPerpHomPic}
There is an exact sequence
\[
	0 \to k^{*} \prod_{x\in X} N_{x}  \to  
	\big((\Sigma_{X}^{*})^{\pperp} \cap \bar V_{0} \big)   \overset{\bar\pi} \to 
	\hom_{\text{gr}}\big(\Pic^0(X),k^{*}\big)  . 
\]
Furthermore, if $k$ is either finite or algebraically closed, 
the map $\bar{\pi}$ is surjective.
\end{cor}

\begin{proof}
Recall that $(\Sigma_{X}^{*})^{\perp} \subset\I_{X}^{1}$ since $k^{*}\subset \Sigma_{X}^{*}$, and thus the map $(\Sigma_{X}^{*})^{\perp}\to \Div(X)$ takes values in $\Div^0(X)$, with $(\Sigma_{X}^{*})^{\perp} \cap \A_{X}^+ = \ker\big((\Sigma_{X}^{*})^{\perp}\to \Div^0(X)\big)$. Accordingly, we have
\[
	(\Sigma_{X}^{*})^{\pperp} \cap \bar V_{0} \,=\, \ker\big((\Sigma_{X}^{*})^{\pperp}\to \Div^0(X)\big).
\]
Recalling Theorem~\ref{T:rad}, equation \eqref{E:SigmaPPerp} and Proposition~\ref{prop:ExSeqOrto}, one obtains the exact sequence
\[
	0 \to  \prod_{x\in X} N_{x}  \to  
	(\Sigma_{X}^{*})^{\pperp} \cap \bar V_{0}   \overset{\pi} \longrightarrow 
	\hom_{\text{gr}}\big(\Pic(X),k^{*}\big).
\]
The inclusion $\Pic^0(X)\subset \Pic(X)$ yields a restriction map $\rho$ which fits into the exact sequence
\[
	0 \to k^* \to \hom_{\text{gr}}\big(\Pic(X),k^{*}\big) \to \hom_{\text{gr}}\big(\Pic^0(X),k^{*}\big) \to 0,
\]
where $a\in k^*$ is mapped to the group homomorphism  $\Pic(X) \to k^{*}$ by defining $\alpha \mapsto a^{\deg D(\alpha)}$ (see \eqref{eq:IXaDivisor}).  The surjectivity of the restriction map on the right-hand side of the previous sequence follows from the existence of a rational point. 

The surjectivity of the sequence in the statement is proven as in Proposition~\ref{prop:ExSeqOrto}.
\end{proof}

\begin{thm}
\label{thm:EXACTSEQSigma}
Let $k$ be either finite or algebraically closed. There is an exact sequence
\begin{equation}
\label{E:EXACTSEQSigma}
	0 \to \hom_{\text{gr}}\big(\Pic^0(X),k^{*}\big)  \to  
	(\Sigma_{X}^{*})^{\pperp}/\big( \Sigma_{X}^{*} \prod_{x\in X} N_{x}\big)  \to  
	\Pic^0(X).	
\end{equation}
Moreover, for $k=\C$, the last arrow is surjective.
\end{thm}

\begin{proof}
Using the map~\eqref{eq:IXaDivisor} $(\Sigma_{X}^{*})^{\perp}\to \Div(X)$ that sends an idele $\lambda$ to its associated divisor, and the exact cohomology sequence of
\[
	0\to  \O_{X}^{*} \to  \Sigma_{X}^{*}\to  \Sigma_{X}^{*}/\O_{X}^{*}\to  0,
\]
one obtains the commutative diagram
\[
	\xymatrix{
	0 \ar[r] & \Sigma_{X}^{*} \ar[r] \ar@{->>}[d] & (\Sigma_{X}^{*})^{\perp} \ar[r] \ar[d] & (\Sigma_{X}^{*})^{\perp}/\Sigma_{X}^{*} \ar[r]\ar[d] & 0
	\\
	0 \ar[r] & \Sigma_{X}^{*}/k^{*} \ar[r] & \Div(X) \ar[r] & \Pic(X) \ar[r] & 0   .
	}
\] 
This diagram induces the following one
\[
	\xymatrix{
	0 \ar[r] & \Sigma_{X}^{*} \ar[r] \ar@{->>}[d] & (\Sigma_{X}^{*})^{\pperp} \ar[r] \ar[d] & (\Sigma_{X}^{*})^{\pperp}/\Sigma_{X}^{*} \ar[r]\ar[d] & 0
	\\
	0 \ar[r] & \Sigma_{X}^{*}/k^{*} \ar[r] & \Div^0(X) \ar[r] & \Pic^0(X) \ar[r] & 0
	 . }
\] 
Looking at the kernels of the vertical arrows, one gets
\[
	0\to k^{*} \to  \ker\big((\Sigma_{X}^{*})^{\pperp}\to \Div^0(X)\big)    \to    
	(\Sigma_{X}^{*})^{\pperp}/\Sigma_{X}^{*} \to  \Pic^0(X),
\]
and the conclusion follows by Corollary~\ref{cor:KerSigmaPerpHomPic}. 

Now consider $k = \C$ and let us prove the surjectivity. Let $X$ be a compact Riemann surface of genus $g$ and $\Sigma$ its field of meromorphic functions. Given a meromorphic function $f \in \Sigma^{*}$ with divisor $(f) = \sum_{i} n_{i} a_{i} - \sum_{j} m_{j} b_{j}$, where $n_{i},m_{j}$ are non-negative integers and $a_{i}, b_{j} \in X$ are pairwise distinct points, it is well-known that there exists $\lambda\in {\mathbb C}^{*}$ such that
\[
	  f(z)
	= \lambda\cdot \frac{\prod_{i} \E(z,a_{i})^{n_{i}}}{\prod_{j} \E(z,b_{j})^{m_{j}}}
	\quad \text{ where }
	\E(z, z')= 
	\begin{cases}
	(z-z')       & \text{ if } g=0, 
	             \\
	\sigma(z-z') & \text{ if } g=1,
	             \\
	E(z,z')      & \text{ if } g=2.
	\end{cases}
\]
Here $\sigma$ is Weierstrass' sigma function and $E$ is the prime form (see \cite{Fay}).
	
Let $\tilde X$ be the universal cover of $X$. By Riemann's uniformization theorem,
\[
	\tilde X= 
	\begin{cases}
	{\mathbb{P}}_{\mathbb C}^{1}, & \text{the Riemann sphere, if $g=0$,}
	                             \\ 
	{\mathbb C},                 & \text{the complex plane, if $g=1$,}
	                             \\
	{\mathbb D},                 & \text{the open unit disk, if $g \geq 2$.}
	\end{cases}
\]
Recall that a meromorphic section of a line bundle on $X$ is a meromorphic function  on $\tilde X$ such that its transformation along the homology cycles of $X$ are given by the automorphic factors of the line bundle. Thus $\E(\tilde z,\tilde z')$ is a bimeromorphic function on $\tilde X \times \tilde X$ satisfying
\begin{equation}
\label{e:Eimpar}
\E(\tilde z,\tilde z')  = - \E(\tilde z',\tilde z).
\end{equation}
Let us fix a set-theoretic section $s$:
\[
	\xymatrix{\tilde X \ar[r]^{\pi} & X \ar@/^/[l]^{s}}.
\]
Given a pair of distinct points $a,b\in X$, we define an idele $\alpha_{ab}$ by
\[
	\alpha_{ab} = ((\alpha_{ab})_{x})_{x \in X},
	\text{ where }
	(\alpha_{ab})_{x}
	\text{ is the germ of }
	\frac{\E(\tilde z, s(a))}{\E(\tilde z, s(b))}\text{ at $s(x)\in \tilde X$}.
\]	
Observe that the divisor of this idele is $a-b\in \operatorname{Pic}^0(X)$, and these are generators. We now show that
\[
	\alpha_{ab} \in (\Sigma^{*})^{\perp},
	\text{ i.e., }
	\globalsymbol{f,\alpha_{ab}} = 1 \quad \forall f\in \Sigma^{*}.
\]
Since $\globalsymbol{}$ is multiplicative, it suffices to consider the following two cases.

\noindent \textbf{Case 1:} $\supp (f)\cap \supp (\alpha_{ab}) = \emptyset$. Then
\[
	\globalsymbol{f,\alpha_{ab}} = \prod_{i} \alpha_{ab}(a_{i})^{-n_{i}} \prod_{j} \alpha_{ab}(b_{j})^{m_{j}} f(a) f(b)^{-1}.
\]
Since $f(a),f(b)$ are meromorphic functions on $X$ and thus their values do not depend on the choice of elements on the fibers of $\pi$, the above equation can be expressed as
\[
	\Big(\frac{\E(s(a_{i}), s(a))}{\E(s(a_{i}), s(b))}\Big)^{-n_{i}}
	\Big(\frac{\E(s(b_{j}), s(a))}{\E(s(b_{j}), s(b))}\Big)^{m_{j}}
	\Big(\lambda\cdot \frac{\prod_{i} \E(a,a_{i})^{n_{i}}}{\prod_{j} \E(a,b_{j})^{m_{j}}}\Big)
	\Big(\lambda\cdot \frac{\prod_{i} \E(b,a_{i})^{n_{i}}}{\prod_{j} \E(b,b_{j})^{m_{j}}}\Big)^{-1}.
\]
Recalling~\eqref{e:Eimpar}, this equals
\[
	(-1)^{2(\sum_{i} n_{i} + \sum_{j} m_{j})}  = 1.
\]

\noindent \textbf{Case 2:} $\supp (f)\cap \supp (\alpha_{ab}) = \{a\}$. Set $a_{1}=a$ and thus $v_{a}(f)=n_{1}$. Then
\begin{align*}
	   \globalsymbol{f,\alpha_{ab}}
	&= (-1)^{n_{1}} 
	   \cdot \Big(\frac{f}{\alpha_{ab}^{n_{1}}}\Big)(a)
	   \cdot \frac{1}{f(b)}
	   \cdot \prod_{i\neq 1} \frac{1}{\alpha_{ab}(a_{i})^{n_{i}}}
	   \cdot \prod_{j} {\alpha_{ab}(b_{j})^{m_{j}}}
	\\
	&= (-1)^{n_{1}} \left(\frac{\lambda \frac{\prod_{i} \E(z,a_{i})^{n_{i}}}{\prod_{j} \E(z,b_{j})^{m_{j}}}}{\big(\frac{\E(\tilde z,s(a))}{\E(\tilde z,s(b))}\big)^{n_{1}}}\right)(a)
	\cdot \Big( \lambda \frac{\prod_{i} \E(b,a_{i})^{n_{i}}}{\prod_{j} \E(b,b_{j})^{m_{j}}} \Big)^{-1}
	\\
	&\qquad \cdot \prod_{i\neq 1} \Big(\frac{\E(s(a_{i}),s(a))}{\E(s(a_{i}), s(b))}\Big)^{-n_{i}}
	\cdot \prod_{j} \Big(\frac{\E(s(b_{j}),s(a))}{\E(s(b_{j}), s(b))}\Big)^{m_{j}}
	\\
	&= (-1)^{n_{1}} \Big(\frac{\lambda  \prod_{i\neq 1} \E(a,a_{i})^{n_{i}} \E(s(a),s(b))^{n_{1}}}
	{\prod_{j} \E(a,b_{j})^{m_{j}}}\Big)
	\cdot 
	\Big( \frac{\prod_{j} \E(b,b_{j})^{m_{j}}}{ \lambda \prod_{i} \E(b,a_{i})^{n_{i}}}\Big)
	\\ 
	&\qquad \cdot \prod_{i\neq 1} \Big(\frac{\E(s(a_{i}),s(a))}{\E(s(a_{i}), s(b))}\Big)^{-n_{i}} \cdot
	\prod_{j} \Big(\frac{\E(s(b_{j}),s(a))}{\E(s(b_{j}), s(b))}\Big)^{m_{j}}
	\\
	&= (-1)^{n_{1}+\sum_{j} m_{j} +\sum_{i\neq 1} n_{i} +\sum_{j} m_{j} +\sum_{i} n_{i}  }   = 1.
\qedhere	
\end{align*}
\end{proof}

\begin{cor}
If $k={\mathbb C}$ and $X$ is an elliptic curve, then
\[
	(\Sigma_{X}^{*})^{\pperp}/\Sigma_{X}^{*} \neq  0.
\]
\end{cor}

\begin{proof}
Under these hypothesis, we may represent $X$ as a Tate curve; that is, the group of $\C$-valued points of $X$ is isomorphic to ${\mathbb C}^{*}/q^{\mathbb Z}$ for a non-zero complex number with $|q| < 1$. Then the map
\[
	  z \longmapsto e^{2\pi i \frac{\log |z|}{\log|q|}}
    : {\mathbb C}^{*}/q^{\mathbb Z} \to {\mathbb C}^{*} 
\]
yields a non-trivial group homomorphism
\[
	\Pic^0(X)\simeq X \simeq {\mathbb C}^{*}/q^{\mathbb Z}  \to{\mathbb C}^{*}. 
\]
We conclude by Theorem~\ref{thm:EXACTSEQSigma}.
\end{proof}

\begin{rem}
\label{R:structQuot}
If $k$ is algebraically closed, then the norm kernels $N_{x}$ are trivial and the exact sequence~\eqref{E:EXACTSEQSigma} of Theorem~\ref{thm:EXACTSEQSigma} is the sequence we mentioned in the introduction. On the other hand, for finite $k$,~\eqref{E:EXACTSEQSigma} and Corollary~\ref{cor:profinite} completely determine the structure of the quotient $(\Sigma_{X}^{*})^{\pperp} / \Sigma_{X}^{*}$.
\end{rem}

\begin{thm}\label{T:SigmaAlg}	
Let $k$ be algebraically closed and let $\Omega$ be a field with  $\Sigma_{X}^{*}\subset \Omega^{*}\subset (\Sigma_{X}^{*})^{\perp}$. Then:
\begin{enumerate}

\item $\Omega^{*} $ is discrete in $\I_{X}$.

\item The natural map $\Omega^{*}\to  (\Sigma_{X}^{*})^{\pperp}$ is injective and its image is discrete.

\item There is a natural injection
\[
	\Omega^{*}/\Sigma_{X}^{*} \hookrightarrow \Pic^0(X).
\]

\end{enumerate}
\end{thm}

\begin{proof}
Observe that $\Omega^{*}\subset (\Sigma_{X}^{*})^{\perp}\subset \I_{X}^1$. Thus, by Theorem~\ref{T:equivalences}, $\Omega^{*} $ is discrete. The kernel of the map $\Omega^{*}\to  (\Sigma_{X}^{*})^{\pperp}$ is $\Omega^{*}\cap \prod(1+\m_x)$, which is equal to $\{1\}$ since $\Omega^{*}\cap \A_{X}^+ =k^{*}$ (again by Theorem~\ref{T:equivalences}). Let us check that its image is discrete. By Proposition~\ref{prop:PropertiesSigmaOrt}, it suffices to see that there exists an effective divisor $D$ such that $\Omega^{*} \cap p\bigl(\prod_{x \in \supp(D)} \{1\} \times \prod_{x \notin \supp(D)} k(x)^{*}\bigr) = \{1\}$ in $(\Sigma_{X}^{*})^{\pperp}$. Let $\omega$ be an element in this intersection. Then $\omega - 1$ has zeros in $\supp(D)$ and no poles. Since $\Omega \subseteq (\Sigma_{X}^{*})^{\perp} \subseteq \I_{X}^{1}$ is a field, this implies $\omega = 1$.

Finally, the exact sequence \eqref{E:EXACTSEQSigma} yields a map $\Omega^{*}/\Sigma_{X}^{*} \to \Pic^0(X)$. Let $\omega\in \Omega^{*}$ be an element in the kernel. After replacing $\omega$ by $f\omega$ for a suitable $f\in \Sigma_{X}^{*}$, we may assume that the divisor of $\omega$ is zero. If $\omega\notin k$, then $t-\omega\in \Omega^{*}$ for all $t\in k$. Observing that $k[\omega]\subseteq \Omega\cap \A_{X}^+$ and recalling that $\dim_{k} \Omega\cap\A_{X}^+ <\infty$, it follows that $\omega\in k$ since $k$ is algebraically closed. 
\end{proof}

	




\begin{thebibliography}{MM}



\bibitem{AP} Anderson, G. W.; Pablos Romo, F.,
\textit{Simple Proofs of Classical Explicit Reciprocity Laws on
Curves using Determinant Grupoids over an Artinian Local Ring},
Comm. Algebra \textbf{32(1)} (2004), pp- 79--102.

\bibitem{ACK} Arbarello, E.; de Concini, C.; Kac, V.G., \textit{The
Infinite Wedge Representation and the Reciprocity Law for
Algebraic Curves}, Proc. of Symposia in Pure Mathematics, Volume
\textbf{49} Part I, A.M.S., (1989), pp.171--190.

\bibitem{ArtinW} Artin, E.; Whaples, G., \textit{Axiomatic characterization of fields by the product formula for valuations}, Bull. Am. Math. Soc. \textbf{51 } (1945), pp. 469--492.

%
%


\bibitem{BryDeligne} Brylinski, J. L.; Deligne, P., \textit{Central extensions of reductive groups by K2}, 
Publ. Math. Inst. Hautes Études Sci. No. \textbf{94} (2001), 5--85. 


\bibitem{CaFr} Cassels, J. W. S.; Fröhlich, A., \textit{Algebraic Number Theory} (Proc. Instructional Conf., Brighton, 1965) Thompson, Washington, D.C.

\bibitem{Fay} Fay, J.D. \textit{Theta functions on Riemann surfaces.}
Lecture Notes in Mathematics \textbf{352.} Springer-Verlag, Berlin-New York (1973).



\bibitem{MP} Mu\~noz Porras, J.M., Pablos Romo, F., \textit{Generalized reciprocity laws}, Trans. Amer. Math. Soc. \textbf{360 (7) } (2008), pp. 3473--3492.




\bibitem{Ono} Ono, T., \textit{On Algebraic Groups Defined by Norm Forms of Separable Extensions},  Nagoya Math. J. \textbf{11} (1957), pp. 125--130.



\bibitem{Pa} Pablos Romo, F., \textit{On the Tame Symbol of an
Algebraic Curve}, Comm. Algebra \textbf{30} (9) (2002), pp. 4349--4368.


\bibitem{Colle}  Plaza Martín, F. J., \textit{Arithmetic infinite Grassmannians and the induced central extensions},  Collect. Math. \textbf{61} (2010), no. 1, 107--129





\bibitem{Se} Serre, J.P., \textit{Groupes alg\'ebriques et corps de classes},
Publications de l'institut de math\'ematique de l'universit\'e de Nancago, VII. Hermann, Paris, Publications Mathematiques \textbf{61} (1958).


\bibitem{Ta} Tate, J. T., \textit{
Residues of Differentials on Curves}, Ann. Scient. \'Ec. Norm. Sup., 4a s\'erie \textbf{1} (1968), pp. 149--159.


\bibitem{TaSym} Tate, J. T., \textit{Symbols in arithmetic}, Actes du Congrès International des Mathématiciens (Nice, 1970), Tome 1,  Gauthier-Villars, Paris (1971), pp. 201--211.


\bibitem{WeilBas} Weil, A., \textit{Basic number theory}, Reprint of the second (1973) edition. Classics in Mathematics. Springer-Verlag, Berlin, 1995, ISBN: 3-540-58655-5 

\bibitem{Weil} Weil, A., \textit{G\'en\'eralisation des fonctions ab\'eliennes},
J. Math. Pures et Appl. \textbf{17} (1938), pp. 47--87.

\bibitem{Weil-corps} Weil, A. \textit{Sur la théorie du corps de classes}, J. Math. Soc. Japan \textbf{3}, (1951). 1--35. (Reviewer: G. Hochschild) 10.0X


\end{thebibliography}
\end{document}